\documentclass[12pt,reqno]{amsart}
\usepackage{amssymb}
\usepackage{graphicx}
\usepackage{amsmath}
\usepackage{a4wide}
\usepackage{version}
\usepackage{enumerate}

 \theoremstyle{plain}
\newtheorem{Thm}{Theorem}
\newtheorem{Lem}{Lemma}

\newtheorem{Cor}{Corollary}
\newtheorem{Prop}{Proposition}

\def\max{\operatorname{max}}
\def\exp{\operatorname{exp}}

\def\l{\ell}
\def\Z{\mathbb Z}

\def\R{\mathbb R}
\def\bee{\begin{enumerate}}
\def\ene{\end{enumerate}}
\def\bei{\begin{itemize}}\def\eni{\end{itemize}}
\def\beq{\begin{equation}}\def\enq{\end{equation}}
\def\beqs{\begin{equation*}}\def\enqs{\end{equation*}}

\def\C{\mathbb C}

\def\R{\mathbb R}

\def\Z{\mathbb Z}

\def\phi{\varphi}

\def\r{\mathbf r}\def\z{\mathbf z}\def\0{\mathbf 0}
\def\j{\mathbf j}\def\v{\mathbf v}

\def\max{\operatorname{max}}

\begin{document}

\title[Mahler measures of polynomials]{Mahler measures of polynomials that are sums\\ of a bounded number of monomials }
\author{ Edward Dobrowolski}
\address{Department of Mathematics and Statistics\\University of Northern British Columbia\\Prince George, BC\\Canada
}
\email{edward.dobrowolski@unbc.ca}
\author{ Chris Smyth}
\address{School of Mathematics and Maxwell Institute for Mathematical Sciences\\
University of Edinburgh\\
Edinburgh EH9 3FD\\
Scotland, U.K.}
\email{c.smyth@ed.ac.uk}
\date{17 August 2016}

\begin{abstract}
We study Laurent polynomials in any number of variables that are sums of at most $k$ monomials.
We first show that the Mahler measure of such a polynomial is at least $h/2^{k-2}$, where $h$ is the height of the polynomial. Next, restricting to such polynomials having integer coefficients, we show that the set of logarithmic Mahler measures of the elements of this restricted set is a closed subset of the nonnegative real line, with $0$ being an isolated point of the set.
 In the final section, we discuss the extent to which such an integer polynomial of Mahler measure $1$ is determined by its $k$ coefficients.
\end{abstract}

\maketitle

\section{Statement of results}

For a polynomial $f(z)\in\C[z]$, we denote by $m(f)$ its {\it logarithmic Mahler measure}
\beq\label{E-1}
m(f)=\int_0^1 \log|f(e^{2\pi it})|\,dt,
\enq
and write $M(f)=\exp(m(f))$ for the (classical) {\it Mahler measure} of $f$. Although first defined by D.H. Lehmer \cite{L}, its systematic study was initiated by Kurt Mahler \cite{Ma60,M,Ma62}.

 Let $h(f)$ denote the {\it height} of $f$ (the maximum modulus of its coefficients). Our first result relates these two quantities.
\begin{Thm}\label{T-1} For an integer $ k \geq 2,$  let
\beq\label{E-f}
f(x) =  a_1z^{n_1} + \cdots + a_{k-1}z^{n_{k-1}} + a_k\in\C[z] \text{ with } n_1>n_2>\cdots>n_{k-1}>0
\enq
 be a nonzero polynomial.  Then
\[
M(f)\geq \frac{h(f)}{2^{k-2}}.
\]
\end{Thm}
The example $(z+1)^{k-1}$ shows that the constant $1/2^{k-2}$ in this inequality cannot be improved to any number  bigger than $1/{k-1 \choose {\lfloor(k-2)/2\rfloor}}$ (which $\sim \sqrt{2\pi k}/2^k$ as $k\to\infty$).

The inequality for the special case $(n_1,n_2\dots,n_{k-1})=(k-1,k-2,\dots,1)$ (i.e., a polynomial of degree $k-1$) follows from a result of Mahler \cite[equation (6)]{Ma62}.

In the other direction we have from \eqref{E-1} the trivial bound $M(f)\le kh(f)$.

\begin{Cor}\label{C-1} Given $k\ge 1$, there are only finitely many possible choices for integers $a_1,\dots,a_k$ such that $M(f)=1$ for some $f(x) =  a_1z^{n_1} +\dots + a_{k-1}z^{n_{k-1}} + a_k$ and any choice of distinct integer exponents $n_1,n_2,\dots,n_{k-1}.$
\end{Cor}

This corollary  leaves open the question of whether, for fixed $a_1,\dots,a_k$, the number of choices for the exponents $n_i$ is finite or infinite. This is discussed in Section \ref{S-exp}.

Theorem \ref{T-1} in fact holds for Laurent polynomials in several variables, as the next result states. Since it follows quite easily from the one-variable case, we decided to relegate this general case to a corollary. Recall that the logarithmic Mahler measure in the general case is defined for $F=F(z_1,\dots,z_\ell)$ as
\begin{equation}\label{E-2}
m(F)=\int_0^1\cdots\int_0^1\log|F(e^{2\pi it_1},\dots,e^{2\pi it_\ell})|\,dt_1\cdots dt_\ell
\end{equation}
Again, $M(F):=\exp(m(F)).$

In \cite{BoSpec}, David Boyd studied the set $\mathcal L$ of Mahler measures of polynomials $F$ in any number of variables having integer coefficients. He conjectured that $\mathcal L$ is a closed subset of $\R.$ Our Theorem \ref{T-2} below is a result in the direction of this conjecture, but where we restrict the polynomials $F$ under consideration to be the sum of at most $k$ monomials.
In \cite[Theorem 3]{Snew}, the second author proved another restricted closure result of this kind, where the restriction was, instead, to integer polynomials $F$ of bounded length (sum of the moduli of its coefficients).

Boyd's conjecture is a far-reaching generalisation of a question of D.H. Lehmer \cite{L}, who asked whether there exists an absolute constant $C>1$ with the property that, for integer polynomials $f$ in one variable, either $M(f)=1$ or $M(f)\ge C.$

We now state our generalisation of  Theorem \ref{T-1}. In it, we write $\z_\l=(z_1,\ldots,z_\l)$.
\begin{Cor}\label{C-2} Let $F(\z_\l)\in\C[\z_\l]$ be a nonzero Laurent polynomial in $\l\ge 1$ variables that is the sum of $k$ monomials. Then
\[
M(F)\geq \frac{h(F)}{2^{k-2}}.
\]
\end{Cor}

Corollary \ref{C-2} is an essential ingredient in our next result. For this, we fix $k\ge 1$ and consider the set ${\mathcal H}_k$ of Laurent polynomials $F(\z_\l)=F(z_1,\dots,z_\l)$ for all $\l\ge 1$ with integer coefficients that are the sum of at most $k$ monomials. So such an $F$ is of the form
\[
F(\z_\l)=\sum_{\j\in J} c(\j)\z_\l^{\j},
\]
 where $J\subset\Z^\l$ has $k$ column vector elements $\j$, with  $\z_{\l}^{\j}=z_1^{j_1}\cdots z_\l^{j_\l}$, where $\j=(j_1,\ldots,j_\l)^{\text{tr}}$,
and the $c(\j)$'s are integers, some of which could be $0$. The number of variables $\l$ defining $F$ is unspecified, and can be arbitrarily large.
We let $m({\mathcal H}_k)$ denote the set $\{m(F)\, :\, F\in {\mathcal H}_k\}.$

\begin{Thm}\label{T-2} The set $m({\mathcal H}_k)$ is a closed subset of $\mathbb R_{\ge 0}$.
Furthermore, $0$ is an isolated point of $m({\mathcal H}_k)$.
\end{Thm}

In fact the isolation of $0$ in $m({\mathcal H}_k)$ has been essentially known for some time, indeed with explicit lower bounds for the size of the gap between $0$ and the rest of the set. The first such bound was given for one-variable polynomials by Dobrowolski, Lawton and Schinzel \cite{DLS}. This was improved by Dobrowolski in \cite{D1} and later improved further in \cite{D2}, where it was shown that for noncyclotomic $f\in\Z[z]$
\[
M(f)\ge  1+\frac1{\exp(a3^{\lfloor(k-2)/4\rfloor}k^2\log k)},
\]
where $a<0.785.$ Arguing as in the proof of Corollary \ref{C-2} below shows that the gap holds for polynomials in several variables too, and so applies to all $m(F)$ in $m({\mathcal H}_k)\!\setminus\!\{0\}.$

\section{Proof of Theorem \ref{T-1} and Corollary \ref{C-1}}\label{S-pf1}

We first prove the theorem by induction under the restriction that all $a_1,\dots,a_k$ are assumed to be nonzero. We employ two well-known facts:
\begin{enumerate}[(i)]
\item{ $M(f)=M(f^*)$ where $f^*(z)=z^{n} f(z^{-1}),$  with $n=\deg f.$ This immediately follows from \eqref{E-1}. }
\item{$M(f)\geq M(\frac{1}{n}f').$ This was proved by Mahler in \cite{M} .}
\end{enumerate}
 For the base case $k=2$ of our induction, we have \mbox{$M(f)=\max\{|a_1|,|a_2|\}= h(f),$ }
as required.

 Suppose now that the conclusion of the (restricted) theorem is true for some $k\geq 2,$ and suppose that $f$ has $k+1$ nonzero terms,  that is, $f(x) =  a_1z^{n_1} + \dots + a_{k}z^{n_{k}} + a_{k+1}.$ Then
$f^*(z) =  a_{k+1}z^{n_1} + a_kz^{n_1-n_k}+\dots + a_{1}.$  Because the $a_i$ are assumed nonzero, both $f$ and $f^*$ have degree $n_1.$
 Suppose that $h(f)=|a_i|$ for some $i,$ $(1\leq i\leq k+1).$ Then $h(\frac{1}{n_1}f')\geq \frac{n_i}{n_1}h(f),$ and $h(\frac{1}{n_1}(f^*)')\geq \frac{n_1-n_i}{n_1}h(f).$ Clearly $\max\{\frac{n_i}{n_1},\frac{n_1-n_i}{n_1}\}\geq \frac{1}{2},$ with $f'$  and $(f^*)'$ having $k$ terms each. Hence, by (i), (ii) and the induction hypothesis
\[
M(f)\ge\max\left\{M\left(\tfrac{1}{n_1}f'\right), M\left(\tfrac{1}{n_1}(f^*)'\right)\right\}\geq \tfrac{1}{2}\frac{h(f)}{2^{k-2}},
\]
which completes the inductive step, and the induction argument.

Now we can do the general case. If some of the $a_i$ can be $0$, then $f(z)$ is of the form $z^jf_1(z)$, where $j\ge 0$ and $f_1$ is of the form \eqref{E-f}, but with $k_1$ nonzero terms, where $0<k_1\le k.$ Then, using \eqref{E-1},
\[
M(f(z))=M(z^jf_1(z))=M(f_1(z))
\]
and, since $h(f_1)=h(f)$ we have
\[
M(f)=M(f_1)\ge \frac{h(f_1)}{2^{k_1-2}}\ge \frac{h(f)}{2^{k-2}}.
\]


\

Corollary \ref{C-1} now follows straight from the theorem, because any such $f$ must have height at most $2^{k-2}$, giving at most $(2^{k-1}+1)^k$ possible choices for $a_1,\dots,a_k$.

\section{ Proof of Corollary \ref{C-2}}\label{S-pf1.5}

For the Proof of Corollary \ref{C-2}, we need the following simple result.

\begin{Lem}\label{L-2} Let $\r_n=(1,n,n^2,\dots,n^{\l-1})\in\Z^\l$. Then for any finite set $V$  of nonzero vectors in $\R^\l$ there is an integer $N$ such that for each $n>N$ the vector $\r_n$ is not orthogonal to any vector $\v\in V$.
\end{Lem}
\begin{proof} Write $\v\in V$ in the form  $\v=(v_1,\dots,v_j,0,\dots,0)$ say, where $v_j\ne 0$ and $j\le \l$. If $j=1$ then $|\v\cdot \r_n|=|v_1|> 0$, so assume $j\ge 2.$ Then
\begin{align*}
|\v\cdot \r_n|=&\left|\sum_{i=1}^jv_i n^{i-1}\right|\\
\ge&|v_j|\left(n^{j-1}-n^{j-2}\left(\sum_{i=1}^{j-1}|v_i/v_j|\right)\right)\\
>\,& 0 \text{ for } n>\sum_{i=1}^{j-1}|v_i/v_j|,=N_\v \text{ say }.
\end{align*}
Now take $N=\max_{\v\in V} N_\v$.
\end{proof}

Following \cite{Sc}, given a fixed integer $s\ge 1$ and a polynomial $F$ in $s$ variables,  $\l\ge 0$ and an $\l\times s$ matrix $A=(a_{ij})\in\Z^{\l\times s}$, define the $s$-tuple $\z_\l^A$ by
\[
\z_\l^A= (z_1,\dots,z_\ell)^A=(z_1^{a_{11}}\!\cdots z_\ell^{a_{\ell 1}},\ldots,z_1^{a_{1s}}\!\cdots z_\ell^{a_{\ell s}})
\]
(which is $(1,1,\dots,1)\in\Z^s$ when $\l=0$) and $F_A(\z_\l)=F(\z_\l^A)$, a polynomial in $\ell$ variables $z_1,\ldots,z_\ell$. Then $m(F_A)$ is defined by \eqref{E-2} with $F$ replaced by $F_A$. Denote  by $\mathcal P(F)$ the set $\{F_A\, : \, A\in\Z^{\l\times s}, \l\ge 0\}$, and by $\mathcal M(F)$ the set $\{m(F_A)\, : \, F_A\in\mathcal P(F), F_A\ne 0 \}.$

In the case $\l=1$, and with $A$ replaced by $\r=(r_1,\dots,r_s)$, we have $z^{\r}=(z^{r_1},\dots,z^{r_s})$ and $F_{\r}(z)=F(z^{\r}).$

We also need the following.

\begin{Prop} \label{P-1}
Let $\l\ge 1$, $n\ge1,$ and $\r_n=(1,n,n^2,\dots,n^{\l-1})$, as in Lemma \ref{L-2}. Then for any Laurent polynomial $F(\z_\l)$ in $\l$ variables $\z_\l=(z_1,\dots,z_\l)$ we have $m(F_{\r_n}(z))\to m(F(\z_\l))$ as $n\to\infty.$ Furthermore, for $n$ sufficiently large,  $h(F_{\r_n})=h(F).$
\end{Prop}

\begin{proof} The first part follows from results of Boyd \cite[p. 118]{Bo1} and Lawton \cite{La}; see also  \cite[Lemma 13 and Proposition 14]{Snew}.
Next, note that  $F$ is the sum of $k$ monomials of the form $c(\j)\z_\l^\j$, so that $F_\r$ is the sum of $k$ monomials of the form $c(\j)(z^\r)^\j = c(\j)z^{\r\j}= a_i z^{t_i}$ say, for some $i$, where $\j\in J$ is a column vector.
 We now take $\r=\r_n$, and apply Lemma \ref{L-2} to the set $V$ of all nonzero differences $\j-\j'$ between elements of $J$. The lemma then guarantees that, for $n$ sufficiently large, the $t_i$ are distinct, so that $F_\r$ and $F$ have the same coefficients. In particular, $h(F_\r)=h(F).$
\end{proof}

\begin{proof}[Proof of Corollary \ref{C-2}] This now follows from Theorem \ref{T-1}, using the fact, from Proposition \ref{P-1}, that, for any $\varepsilon>0$, $F$ has the same height and the same number of monomials as some one-variable polynomial $F_\r$ with $|m(F_\r)-m(F)|<\varepsilon$.
\end{proof}

\section{Proof of Theorem \ref{T-2}}\label{S-pf2}

\begin{proof}[Proof of Theorem \ref{T-2}]

 Throughout, $k\ge 2$ is fixed, while $\l\ge 0$ can vary. Take any $F\in{\mathcal H}_k,$ with $F(\z_\l)=\sum_{\j\in J} c(\j)\z_\l^{\j},$ say, where $J$ is a $k$-element subset of $\Z^\l$. Then $F\in\mathcal P(a_1z_1+\cdots+a_kz_k)$ for some integers $a_i$, where $\{c(\j)\}_{\j\in J}=\{a_i\}_{i=1,\dots,k}$ as multisets. (Again, some $a_i$'s could be $0$.)  Conversely, every element of $\mathcal P(a_1z_1+\cdots+a_kz_k)$ is a sum of $k$ monomials. (Note that because monomial terms may combine to form a single monomial term, or indeed vanish, the resulting $a_i$'s for some polynomials in $\mathcal P(a_1z_1+\cdots+a_kz_k)$
may be different from the $a_i$'s that we started with. Because of this, the height of some such polynomials  may be larger or smaller than the height $\max_{i=1}^k |a_i|$ of $a_1z_1+\cdots+a_kz_k$. This does not matter, however!)

Next, take some bound $B>0$ and consider all $F$ such that $m(F)\le B$. Then, by Corollary \ref{C-2},
\[
 h(F)\le 2^{k-2}e^{B},
\]
so that there are only finitely many choices for the integers $a_i$. So $m(F)$ belongs to the union -- call it $U_B$ -- of finitely many sets $\mathcal M(a_1z_1+\cdots+a_kz_k):=m(\mathcal P(a_1z_1+\cdots+a_kz_k)),$ intersected with the interval $[0,B]$. Thus $U_B$ is closed since, by \cite[Theorem 1]{Snew}, each set $\mathcal M(a_1z_1+\cdots+a_kz_k)$ is closed. Note that the finite number of sets comprising $U_B$ depends on $k$ and on $B$, but not on $F$. Finally, we see that $m({\mathcal H}_k)$ is closed. This is because any convergent sequence in $m({\mathcal H}_k)$, being bounded,  belongs, with its limit point, to $U_B$ for some $B$.

Finally, to show that $0$ is an isolated point of $m({\mathcal H}_k)$, note that, by \cite[Theorem 2]{Snew} it is an isolated point of every $\mathcal M(a_1z_1+\cdots+a_kz_k)$ that contains $0$. Hence, since $0\in U_B$ for every $B>0$, it is an isolated point of $U_B$ and therefore also of $m({\mathcal H}_k)$.

\end{proof}

\section{Products of cyclotomic polynomials that have the same coefficients}\label{S-exp}

In this section we address the question of whether two or more integer polynomials having Mahler measure $1$ (and so being products of cyclotomic polynomials $\Phi_n(z)$) can have the same set of $k$ nonzero coefficients. We restrict our attention to the case where all the coefficients $a_i$ are $1$. This already indicates what can happen.

Let $k\geq 2$ and define
\[
S= \{(n_1,\dots,n_{k-1})\in \Z^{k-1} \mid n_1 > n_2 > \dots > n_{k-1} > 0 \text{ with } \gcd(n_1,\dots,n_{k-1})=1\}.
\]

\begin{Prop}\label{P-exp} For $\mathbf{n}\in S$ define $f_{\mathbf{n}}(x) =  z^{n_1} + \dots + z^{n_{k-1}} + 1.$  Let
\[
S_{\text c}=\{\mathbf{n}\in S \mid M(f_{\mathbf{n}})=1\}.
\]
The set $S_{\text c}$ is finite if and only if $k$ is a prime number.
\end{Prop}

 Since for instance $\Phi_5(z)$ and $\Phi_5(z)\Phi_6(z)$ have the same nonzero coefficients, $S_{\text c}$ can, however, contain more than one element for $k$ prime.

\begin{proof} Consider first the case of composite $k.$ \\
Suppose that $k=st,$ where integers $s $ and $t$ are greater than 1. Let $g(z)=\sum_{j=0}^{s-1}z^j$  and $h(x)=\sum_{j=0}^{t-1}z^j.$ If $m$ and $l$ are arbitrary integers greater than 1 and such that $\gcd(m,l)=\gcd(m,t!)=\gcd(l,s!)=1$ then it is not difficult to check that
$g(z^m)h(z^l)=f_{\mathbf{n}}(z)$ for some $\mathbf{n}\in S.$ Since $M(g(z^m))=M(h(z^l))=1$, we see that in fact $f_\mathbf{n}(z)\in S_{\text c}$, and so   $S_{\text c}$ is infinite.

Now suppose that $k=p$ is prime. \\
Let $\mathbf{n}\in S_{\text c}$ and consider $f_{\mathbf{n}}(z) =  z^{n_1} +\dots + z^{n_{k-1}} + 1.$ By Kronecker's Theorem, $f_{\mathbf{n}}$ is a product of cyclotomic polynomials. Further, $f(1)=p.$ However the value of a cyclotomic polynomial at 1 is $\Phi_m(1)=1$ if $m$ is divisible by two distinct primes or $\Phi_{m}(1)=q$ if $m$ is a power of a single prime $q.$  Thus, for some $n,$ $\Phi_{p^n}$ divides $f.$

We claim that $n=1.$ To show this, we use a theorem of Mann \cite{Mann} which, in our notation, takes the form of the following lemma.

\begin{Lem}\label{L-exp}
Let $f(z)=\sum_{i=1}^{k-1} a_iz^{n_i}+a_k\in \Z[z],$ with $(n_1,\dots,n_{k-1})\in S$ and where the coefficients $a_i,$ $1\leq i\leq k$  nonzero. If a cyclotomic polynomial $\Phi$ divides $f$ but $\Phi$ does not divide any proper subsum of $\sum_{i=1}^{k-1} a_iz^{n_i}+a_k$ then $\Phi=\Phi_q,$ where $q$ is squarefree and composed entirely of primes less than or equal  to $k.$
\end{Lem}

In the case of $f_{\mathbf{n}}$, a proper subsum defines a polynomial $g$ such that $g(1)$ counts its number of monomials. Hence $g(1)<p,$ and consequently $g$ cannot be divisible by $\Phi_{p^n}$. By Lemma \ref{L-exp}, $p^n$ is squarefree, so $n=1,$ as claimed.

Next, we need the following result.

\begin{Thm} [{{Dobrowolski \cite[Theorem 2 and Corollary 1]{D2}}}] \label{T-D} Let $f(z)=\sum_{i=1}^k a_iz^{n_i} \in\Z[z]$, $f(0)\ne 0$,
 be a polynomial with $k$ nonzero coefficients. There are positive constants $c_1$ and $c_2$, depending only on $k$,
 and polynomials $f_0,f_2\in\Z[z]$ such that if
\[
\deg f_\text{c} \ge \left(1-\frac{1}{c_1}\right)\deg f
\]
then either
\beq\label{D-1}
f(z)=f_0(z^l), \qquad \text{where } \deg f_0\le c_2,
\enq
or
\beq\label{D-2}
f(z)=\left(\prod_i \Phi_{q_i}(z^{l_i})\right)f_2(z), \text{ where } \min_i\{l_i\}\ge\max\left\{\frac{1}{2c_1}\deg f,\deg f_2\right\}.
\enq
Furthermore in this second case then $f_2(z)=\pm \sum_{i=j}^k a_iz^{n_i}$ for some $j$ with $1< j\le k$.
\end{Thm}
In this theorem $\Phi_q$ is the $q$-th cyclotomic polynomial, while $f_\text{c}$ is the product of all cyclotomic polynomials dividing $f$.

 Now we apply Theorem \ref{T-D}. If equation \eqref{D-1} of its conclusion occurs then, with our restriction on $S,$  $\deg f \leq c_2.$ If equation \eqref{D-2}  occurs, then either $\deg f \leq 2c_1$ or $\min\{l_i\}\geq 2.$ In the latter case $\Phi_p$ must divide $f_2,$  where $f_2$ is a proper subsum of $f.$  Hence $f_2(1)<p$, contradicting $\Phi_p \mid f_2.$ Thus in all admissible cases the degree of $f$ is bounded by a constant depending only on $k.$ Therefore $S_{\text c}$ is finite.
\end{proof}

\end{document}